\documentclass{amsart}
\usepackage{amsthm,stmaryrd,amsmath,amscd}
\usepackage{amssymb}
\usepackage{epsfig}
\usepackage[usenames,dvipsnames]{color}
\usepackage{verbatim}
\usepackage{hyperref,subfig}
\usepackage{dsfont}
\usepackage{upquote}
\usepackage{mathrsfs,float}
\usepackage[all]{xy}
\usepackage[active]{srcltx}
\usepackage[utf8]{inputenc}
\DeclareUnicodeCharacter{1D70E}{{\sigma}}
\DeclareUnicodeCharacter{2297}{{\otimes}}
\DeclareUnicodeCharacter{2206}{{\Delta}}

\newcommand{\RH}{{R^H}}

\newcommand{\I}{\mathcal{I}}

\newcommand{\E}{\mathcal{E}}
\newcommand{\A}{{\mathcal A}}

\newcommand{\trace}{\operatorname{trace}}

\newcommand{\brk}[1]{{\left\langle{#1}\right\rangle}}

\newcommand{\gl}{\ensuremath{\mathfrak{gl}}}

\newcommand{\e}{{\operatorname{e}}}
\newcommand{\slt}{{\mathfrak{sl}(2)}}

\newcommand{\UH}{{U_\ii^{H}\slt}}

\newcommand{\cat}{\mathscr{C}}

\newcommand{\D}{\mathscr{D}}
\newcommand{\cR}{\mathscr{R}}
\newcommand{\Id}{\operatorname{Id}}
\newcommand{\bp}[1]{{\left(#1\right)}}

\newcommand{\End}{\operatorname{End}}
\newcommand{\Ad}{\operatorname{Ad}}

\newcommand{\sign}{\operatorname{sign}}

\newcommand{\C}{\ensuremath{\mathbb{C}} }
\newcommand{\Z}{\ensuremath{\mathbb{Z}} }
\newcommand{\Zq}{\ensuremath{\mathbb{Z}[q,q^{-1}]} }

\newcommand{\wb}{\overline}

\newcommand{\et}{{\quad\text{and}\quad}}
\newcommand{\ets}{{\,,\quad}}

\def\restriction#1#2{\mathchoice
              {\setbox1\hbox{${\displaystyle #1}_{\scriptstyle #2}$}
              \restrictionaux{#1}{#2}}
              {\setbox1\hbox{${\textstyle #1}_{\scriptstyle #2}$}
              \restrictionaux{#1}{#2}}
              {\setbox1\hbox{${\scriptstyle #1}_{\scriptscriptstyle #2}$}
              \restrictionaux{#1}{#2}}
              {\setbox1\hbox{${\scriptscriptstyle #1}_{\scriptscriptstyle #2}$}
              \restrictionaux{#1}{#2}}}
\def\restrictionaux#1#2{{#1\,\smash{\vrule height .8\ht1 depth .85\dp1}}_{\,#2}} 

\newtheorem{theo}{Theorem}[section]
\newtheorem{lemma}[theo]{Lemma}
\newtheorem{prop}[theo]{Proposition}
\newtheorem{cor}[theo]{Corollary}

\theoremstyle{definition}

\newtheorem{rem}[theo]{Remark}

\newtheorem{exo}[theo]{Exercise}

\theoremstyle{remark}

\renewcommand{\qedsymbol}{\fbox{\thetheo}}

\newcounter{exo} \newcounter{numexercice}
\renewcommand{\theexo}{\arabic{exo}}

\begin{document}
\title[Other quantum relatives of the Alexander polynomial]{Other
  quantum relatives of the Alexander polynomial through the
  Links-Gould invariants}

\author[B.M. Kohli]{Ben-Michael Kohli}
\address{IMB UMR5584, CNRS, Universit\'e Bourgogne Franche-Comt\'e, F-21000 Dijon, France}
\email{Ben-Michael.Kohli@u-bourgogne.fr}

\author[B. Patureau]{Bertrand Patureau-Mirand}
\address{UMR 6205, LMBA, Universit\'e de Bretagne-Sud, BP 573, 56017 Vannes, France }
\email{bertrand.patureau@univ-ubs.fr}
\thanks{.
 }\

\begin{abstract}
  Oleg Viro studied in \cite{Vi} two interpretations of the
  (multivariable) Alexander polynomial as a quantum link invariant:
  either by considering the quasi triangular Hopf algebra associated
  to $U_q\slt$ at fourth roots of unity, or by considering the
  super Hopf algebra $U_q\gl(1|1)$.  In this paper, we show these Hopf
  algebras share properties with the $-1$ specialization of
  $U_q\gl(n|1)$ leading to the proof of a conjecture of David De Wit,
  Atsushi Ishii and Jon Links on the Links-Gould invariants.
 \end{abstract}

\maketitle
\setcounter{tocdepth}{3}


\section{Introduction}
The Links-Gould invariants of links $LG^{n,m}$ are two variable
quantum link invariants. They are derived from super Hopf Algebras
$U_q\gl(n|m)$. David De Wit, Atsushi Ishii and Jon Links conjectured
\cite{DWIL} that for any link $L$
$$
LG^{n,m}(L ; \tau,e^{\bold{i} \pi/m}) = \Delta_L(\tau^{2m})^n,
$$
where $\Delta_L$ is the Alexander-Conway polynomial of $L$. They
proved the conjecture when $(n,m) = (1,m)$ and when $(n,m) = (2,1)$
for a particular class of braids. A complete proof of the $(n,1)$ case
for $n = 2,3$ is given in \cite{Ko}. However this is achieved by
studying the invariants at hand at the level of
\emph{representations}, which requires computation of an explicit
$R$-matrix for each $LG^{n,1}$, making that method hard to implement as
$n$ grows.

Here we prove the $(n,1)$ case of the conjecture for any $n$:
$$
LG^{n,1}(L ; \tau,-1) = \Delta_L(\tau^{2})^n.
$$
To do so we study the structure of the \emph{universal} objects
directly, and in particular the (super) Hopf algebras and universal
R-matrices that are involved.

However, the strong version of the conjecture is still open.

\section{Hopf algebras for the Alexander polynomial}
We first define a Hopf algebra $U$ which is an essential ingredient
for the quantum relatives of the Alexander polynomial.  Unfortunately
this algebra is only braided in a weak sense.  Then we recall two 
quantum groups which can be seen as central extensions
of $U$. One was first used by Murakami \cite{jM}, both were
studied by Viro in \cite{Vi}. Finally we compare the braidings of the two Hopf algebras.
\subsection{A braided Hopf algebra $U$}\label{S:U} 
\newcommand{\ii}{\mathbf{i}}
The following Hopf algebra $U$ is a version of quantum $\slt$ when
the quantum parameter $q$ is a fourth root $\ii$ of $1$.  The complex algebra $U$ is presented by generators $k^{\pm1},e,f$ and relations
$$ke+ek=kf+fk=e^2=f^2=0\et ef-fe=k-k^{-1}.$$
The coproduct, counity and antipode of $U$ are given by 
\begin{align*}
  ∆(e)&= 1⊗ e + e⊗ k, 
  &\varepsilon(e)&= 0, 
  &S(e)&=-ek^{-1}, 
  \\
  ∆(f)&=k^{-1} ⊗ f + f⊗ 1,  
  &\varepsilon(f)&=0,& S(f)&=-kf,
    \\
  ∆(k)&=k⊗ k,
  &\varepsilon(k)&=1,
  & S(k)&=k^{-1}.
\end{align*}
This Hopf algebra can be seen in a sense as a a "double" of Bodo Pareigis' Hopf algebra \cite{Par} that would be $<k,f>$ with our notations.
A pivotal structure is a group like element $\phi$ whose conjugation
is equal to the square of the antipode.  There is non obviously a better
choice which is given by $\phi=k^{-1}$.

Let $\tau :x⊗y\mapsto y⊗x$ be the switch of factors. Hopf
algebra $U$ is not quasi-triangular but it is braided in the sense of
\cite{R}: there exists an (outer) algebra automorphism $\cR:U⊗
U\to U⊗ U$ different from $\tau$ that satisfies
\begin{align}
  \cR\circ\Delta&=\tau\circ\Delta,\\
  \Delta_1\circ\cR&=\cR_{13}\cR_{23}\label{eq:deltaR1},\\
  \Delta_2\circ\cR&=\cR_{13}\cR_{12}\label{eq:deltaR2}.
\end{align}

Automorphism $\cR$ admits a regular splitting (see \cite{R}) $\cR=\D\circ\Ad_{\check R}$ where $\Ad_{\check R}$ is the conjugation by the invertible element 
$$\check R=1+e⊗f$$
and $\D$ is an outer automorphism satisfying equations similar to \eqref{eq:deltaR1} and \eqref{eq:deltaR2} and defined by:
$$\D\circ\tau=\tau\circ\D\ets\D(e⊗1)=e⊗k\ets\D(f⊗1)=f⊗k^{-1}\et\D(k⊗1)=k⊗1.$$
The elements $k^{\pm2}$ generate a central sub-Hopf algebra and for
any $g\in\C\setminus\{0,1\}$, the quotient $U/(k^2-g)$ is a
8-dimensional semi-simple Hopf algebra with two isomorphism classes of
irreducible representations $V_{\pm a}$ where $a^2=g$. The
representation $V_a$ is 2-dimensional and can be written in a certain basis $(e_0, e_1)$
\begin{equation}
  \label{eq:repU}
k = \begin{pmatrix}
   a & 0 \\
   0 & -a
\end{pmatrix} ,
e = \begin{pmatrix}
   0 & 1 \\
   0 & 0
\end{pmatrix} ,
f = \begin{pmatrix}
   0 & 0 \\
   a-\frac1a & 0
\end{pmatrix} .
\end{equation}
The central element $ef + fe$ acts by $(a-a^{-1}) I_2$.
\subsection{The $\slt$ model and the Alexander polynomial}\label{S:QUantSL2H} 
From \cite{Vi,CGP3} the $\slt$ model is the unrolled version of
quantum $\slt$ at $q=\ii=\exp(\ii\pi/2)$.  It is an algebra $\UH$ generated by $K^{\pm1},E,F,H$.  Its presentation is obtained from that of $U$ ($U_0=\brk{K^{\pm1},E,F}\simeq U$) by adding the generator $H$ and the following relations:
$$[H,K]=0\ets[H,E]=2E\ets[H,F]=-2F.$$
We will consider the category $\cat$ of weight modules: finite
dimensional vector spaces where element $H$ acts diagonally
and 
\begin{equation}
  \label{eq:KH}
  K=\ii^H=\exp(\ii \pi H/2).
\end{equation}
The pivotal Hopf algebra structure $U$ is extended to $\UH$ by the following relations: \footnote{Compared to Viro, we use
  the opposite coproduct here.}
\begin{align*}
  ∆(H)&=1⊗ H+H⊗ 1
  &\varepsilon(H)&=0,
  & S(H)&=-H.
\end{align*}

As in $U$, the pivotal element is $\Phi=K^{-1}$ so that
$S^2(\cdot)=\Phi\cdot\Phi^{-1}$.  With this pivotal structure,
category $\cat$ is ribbon with braiding given by the switch
$\tau:x⊗y\mapsto y⊗x$ composed with the action of the universal
$R$-matrix:
$$\RH=\ii^{H⊗H/2}(1+E⊗F).$$
\begin{lemma}\label{L:DH}
  For any two representations $V,W\in\cat$, the conjugation in $V⊗W$ by
  $D^H := \ii^{{H⊗H/2}}$ induces an
  automorphism $\D^H$ of $\End_\C(V⊗W)$ which satisfies
  $$\rho_{V⊗W}\circ\D=\D^H\circ\rho_{V⊗W}:U⊗U\to\End_\C(V⊗W).$$
\end{lemma}
\begin{proof}
  This is an easy consequence of Equation \eqref{eq:KH}. More
  generally, if $x,y\in U$ satisfy $[H,x]=2mx$ and $[H,y]=2m'y$, then
  ${H⊗H}.{x⊗y}= {x⊗y}.(H+2m)⊗(H+2m')$ so\\
  $
  \begin{array}{rl}
    \ii^{\rho_{V⊗W}({H⊗H}/2)}\rho_{V⊗W}({x⊗y})&=
    \rho_{V⊗W}({x⊗y})\ii^{\rho_{V⊗W}({(H+2m)⊗(H+2m')}/2)}\\
    &=\rho_{V⊗W}((x⊗K^m)(K^{m'}⊗y))\ii^{\rho_{V⊗W}({H⊗H}/2)}\\
    &=\rho_{V⊗W}(\D({x⊗y}))\ii^{\rho_{V⊗W}({H⊗H}/2)}.
  \end{array}$
\end{proof}
For each complex number $\alpha$ which is not an odd integer, $\UH$
possesses, up to isomorphism, a unique two dimensional irreducible
representation $V_\alpha$ with Spec$(H)=\{\alpha+1,\alpha-1\}$.  Its
restriction to $U$ is representation $V_a$ where
$a=\ii^{\alpha+1}$ and the action of $H$ is given by $ H
= \begin{pmatrix}
  \alpha +1 & 0 \\
  0 & \alpha -1
\end{pmatrix} .
$

In the representation $V_\alpha ⊗ V_\beta$, with respect to
basis $(e_0 ⊗ e_0, e_0 ⊗ e_1, e_1 ⊗ e_0, e_1 ⊗
e_1)$ the braiding is:

$$\textbf{i}^{\frac{\alpha \beta -1}{2}}\begin{pmatrix}
   \textbf{i}^{\frac{\alpha + \beta + 2}{2}} & 0 & 0 & 0 \\
   0 & 0 & \textbf{i}^{\frac{\alpha - \beta}{2}} & 0 \\
   0 & \textbf{i}^{\frac{-\alpha + \beta}{2}} & \textbf{i}^{\frac{-\alpha + \beta}{2}} (\textbf{i}^{\beta +1} - \textbf{i}^{-\beta -1}) & 0 \\
   0 & 0 & 0 & \textbf{i}^{\frac{-\alpha - \beta + 2}{2}}
\end{pmatrix}.$$
In the case where $\alpha = \beta$, the $R$-matrix then takes the particular form

$$\tau \RH = \textbf{i}^{\frac{\alpha^2 -1}{2}}\begin{pmatrix}
  t^{-1/2} & 0 & 0 & 0 \\
   0 & 0 & 1 & 0 \\
   0 & 1 & (t^{-1/2} - t^{1/2}) & 0 \\
   0 & 0 & 0 & -t^{1/2}
\end{pmatrix}$$
where we set $t^{1/2} = \textbf{i}^{-\alpha-1}$. 

The ribbon category we consider here allows us to apply the
Reshetikhin-Turaev theory \cite{RT0} to construct a framed link
isotopy invariant in $S^3$.  It becomes an unframed link isotopy
invariant if one divides the above $R$-matrix on $V_\alpha ⊗
V_\alpha$ by the value of the twist
$\theta_\alpha=\textbf{i}^{\frac{\alpha^2 -1}{2}}$.  In this
particular case, the invariant we find is the Conway normalization of
the classical Alexander polynomial, see \cite{Vi}. The
Links-Gould invariants $LG^{n,1}$ that will interest us in the
following are obtained by the same general construction using other
Hopf algebras.  Explicitly, the Reshetikhin-Turaev functor gives
representations of braid groups $B_\ell$
$$
\begin{array}{rcl}
  \Psi_{V_\alpha^{⊗ \ell}} : B_\ell &\longrightarrow &GL(V_\alpha^{⊗ \ell})\\
  \sigma_i&\mapsto&\Id_{V_\alpha}^{⊗{i-1}} ⊗ \theta_\alpha^{-1}\tau 
  \RH ⊗ \Id_{V_\alpha}^{⊗{\ell-i-1}},
\end{array}
$$ 
where $\sigma_i$ is the $i^{\text{th}}$ standard Artin generator of braid group $B_\ell$.

Let $L$ be an oriented link in $S^3$ obtained as  closure of a braid in $\ell$ strands $b \in B_\ell$. Then: \newline
1) There exists a scalar $c$ such that $\trace_{2,3,...,\ell}((\Id_{V_\alpha} ⊗ (K^{-1})^{⊗ \ell-1}) \circ \Psi_{V_{\alpha}^{⊗{\ell}}}(b))= c.\Id_{V_\alpha}$,\newline
2) $L\mapsto c$ is a link invariant and is equal to the Alexander polynomial of $L$, $\Delta_L(t)$.

\begin{rem}
Identifying algebras $\End_{\mathbb{C}}(V_\alpha^{⊗ \ell})$ and $\End_{\mathbb{C}}(V_\alpha)^{⊗ \ell}$, the partial trace operator is defined by $\trace_{2,3,...,\ell}(f_1 ⊗ ... ⊗ f_\ell) := \trace(f_2) \trace(f_3) ... \trace(f_\ell) f_1 \in \End_{\mathbb{C}}(V_\alpha)$ for any $f_1, ..., f_\ell \in \End_{\mathbb{C}}(V_\alpha)$.
\end{rem}

\subsection{An example of bosonization: the $\gl(1|1)$ model}\label{S:bosonisation} 
\subsubsection{Bosonization}
Here we recall Majid's trick \cite{Ma} to transform a super Hopf
algebra into an ordinary one.

Let $H$ be a pivotal super Hopf algebra and $\cat$ be its even monoidal
category of representations (morphisms are formed by even $H$-linear maps).  Let $H^𝜎$ be the bosonization of $H$: as an algebra, $H^𝜎$ is the semi-direct product of $H$ with
$\Z/2\Z=\{1,𝜎\}$ where the action of $𝜎$ or equivalently the
commutation relations in $H^𝜎$ are given by 
$$\forall x\in H,\,𝜎 x=(-1)^{|x|}x𝜎.$$
The coproduct $∆^𝜎$ on $H^𝜎$ is given by $∆^𝜎
𝜎=𝜎⊗ 𝜎$ and
$$\forall x\in H,\, ∆^𝜎(x)=\sum_i x_i𝜎^{|x_{i}'|}
⊗ x_i'\text{ where }∆(x)=\sum_i x_i⊗ x_i'.$$

If ${R} = \sum_i {R}_i^{(1)} ⊗ {R}_i^{(2)}$ is the universal $R$-matrix in $H$, then the following formula defines a universal $R$-matrix in $H^{\sigma}$:
$$
{R}^{\sigma} = {R}_1 \sum_i {R}_i^{(1)} \sigma^{|{R}_i^{(2)}|} ⊗ {R}_i^{(2)} \text{, where } {R}_1 = \frac{1}{2}(1 ⊗ 1 + \sigma ⊗ 1 + 1 ⊗ \sigma - \sigma ⊗ \sigma).
$$

Given a super representation $V=V_{\wb0}\oplus V_{\wb1}$ of $H$ we get
a representation of $H^𝜎$ by setting
$𝜎_{|V}=\Id_{V_{\wb0}}-\Id_{V_{\wb1}}$. Reciprocally, since
$𝜎^2=1$, every $H^𝜎$-module inherits a natural $\Z/2Z$ grading:
$W$ splits into $W=W_{\wb0}\oplus W_{\wb1}$ where we define
$W_{\wb0}=\ker(𝜎-1)$ and $W_{\wb1}=\ker(𝜎+1)$.
\begin{theo} [Majid Theorem 4.2] The even category 
  of super $H$-modules can be identified with the category of
  $H^𝜎$-modules.
\end{theo}
Remark that the antipode of $H^𝜎$ is given by $x\mapsto 𝜎^{|x|}S(x)$
and if $H$ as a pivot $\phi$ then one can choose $\phi^𝜎=𝜎\phi$ as a pivot in $H^𝜎$.
\subsubsection{The $\gl(1|1)$ model}
\newcommand{\KK}{C}
\newcommand{\EE}{I}
Using the same notations as Viro: $U_q\gl(1|1)$ is the pivotal super Hopf algebra generated by two
odd generators $X, Y$, two even generators ${\EE}$, $G$ satisfying the
relations
$$XY+YX =\dfrac{{\KK} - {\KK}^{-1}}{q-q^{-1}}\ets X^2 = Y^2 = 0,$$
$$[{\EE}, X] = [{\EE}, Y ] = [{\EE}, G ] = 0,$$
$$[G, X] = X\ets[G, Y ] = -Y,$$
where ${\KK} = q^{\EE}$, with coproduct
$$∆({\EE}) = 1 ⊗ {\EE} + {\EE} ⊗ 1\ets ∆(G) = 1 ⊗ G + G ⊗ 1,$$
$$∆(X) = X ⊗ {\KK}^{-1} + 1 ⊗ X\ets ∆(Y) = Y ⊗ 1 + {\KK} ⊗ Y,$$
counit
$$
\varepsilon(X) = \varepsilon(Y) = \varepsilon({\EE}) = \varepsilon(G) = 0,
$$
antipode
$$
S({\EE}) = -{\EE} \text{,  }S(G) = -G \text{,  }S(X) = -X{\KK} \text{,  }S(Y) = -Y {\KK}^{-1},
$$
pivot $$\phi=K$$
and universal $R$-matrix
$$R = (1 + (q - q^{-1} )(X ⊗ Y )({\KK} ⊗ {\KK}^{-1} )) q^{-{\EE}⊗G-G⊗{\EE}}.$$

Its bosonization $U_q\gl(1|1)^𝜎$ contains a sub-Hopf algebra $U_1$
isomorphic to $U$ given by 
$$e=(q - q^{-1})X \sigma\ets f = Y\et k ={\KK}^{-1} \sigma.$$ 
Indeed, these elements satisfy the following:
$$ef-fe=(q - q^{-1})(X \sigma Y - Y X \sigma)=(q - q^{-1})(-XY-YX)𝜎=k - k^{-1},$$
$$ke+ek=kf+fk=0,$$
$$∆^𝜎(e)=(q - q^{-1}) \Delta^{\sigma}(X \sigma) = 
(q - q^{-1}) (X ⊗ {\KK}^{-1} + \sigma ⊗ X) (\sigma ⊗ \sigma) = e ⊗ k + 1 ⊗ e ,$$
$$∆^𝜎(f)= \Delta^{\sigma}(Y) = Y ⊗ 1 + {\KK} \sigma ⊗ Y = f ⊗ 1 + k^{-1} ⊗ f ,$$
$$\Delta^{\sigma}(k) = k ⊗ k.$$
In the bosonization, the universal $R$-matrix is 
$$
R^{\sigma} = R_1 q^{-({\EE} ⊗ G + G ⊗ {\EE})} (1+ e ⊗ f) \text{, where } R_1 = \frac{1}{2}(1 ⊗ 1 + \sigma ⊗ 1 + 1 ⊗ \sigma - \sigma ⊗ \sigma).
$$
\begin{lemma}\label{lem}
  Denoting 
  $D' = q^{- {\EE} ⊗ G - G ⊗ {\EE}}$ and $D^{\sigma} = R_1
  D'$ we have, for any $x,y \in U=U_1$:
$$ R_1 (x ⊗ y) R_1^{-1} = \sigma^{|y|} x ⊗ y \sigma^{|x|}\ets
D' (x ⊗ y) (D')^{-1} = x {\KK}^{-d_G(y)} ⊗ y {\KK}^{-d_G(x)},$$
$$D^{\sigma} (x ⊗ y) (D^{\sigma})^{-1} = ({\KK}^{-1} \sigma)^{d_G(y)} x ⊗ y ({\KK}^{-1} \sigma)^{d_G(x)}=\D(x⊗ y),$$
where 
$d_G(x) \in \mathbb{Z}$ is defined by $[G,x] = d_G(x) x$.
\end{lemma}
\begin{rem}
For a homogeneous $a \in U_0$, $|a| = d_G(a)$ modulo $2$.
\end{rem}
Let us recall a family of $2$-dimensional
$U_q\gl(1|1)^{\sigma}$-modules. This family is parametrized by two
complex numbers $(j,J)$ and 
$\varepsilon \in \{{0}, {1}\}$, 
see \cite{Vi}.  It extends the representation $V_a$ of $U_1$ where
$a=(-1)^{\varepsilon} q^{-2j}$.  Written in matrix form,
$$
\EE = \begin{pmatrix}
   2 j & 0 \\
   0 & 2j
\end{pmatrix} ,
G = \begin{pmatrix}
   \frac{J+1}{2} & 0 \\
   0 & \frac{J-1}{2}
\end{pmatrix} ,
$$
$$
X = \begin{pmatrix}
   0 & \frac{q^{2j}-q^{-2j}}{q-q^{-1}} \\
   0 & 0
\end{pmatrix} ,
Y = \begin{pmatrix}
   0 & 0 \\
   1 & 0
\end{pmatrix} ,
\sigma = \begin{pmatrix}
   (-1)^{\varepsilon} & 0 \\
   0 & -(-1)^{\varepsilon}
\end{pmatrix} .
$$

\subsection{Comparing the actions of $R^{\sigma}$ and $\RH$}
$U_0 \subset \UH$ and $U_1 \subset U_q\gl(1|1)^{\sigma}$ are two isomorphic Hopf algebras. The goal of this paragraph is to show the action of 
$$\RH=\ii^{H⊗H/2}(1+E⊗F) \in \UH ⊗ \UH$$
and that of
$$R^{\sigma} = R_1 q^{-({\EE} ⊗ G + G ⊗ {\EE})} (1+ e ⊗ f) \in U_q\gl(1|1)^{\sigma} ⊗ U_q\gl(1|1)^{\sigma}$$
on two representations $V_1^H ⊗ V_2^H$ of $\UH$ and $V_1^{\sigma} ⊗ V_2^{\sigma}$ of $U_q\gl(1|1)^{\sigma}$ are identical up to a scalar multiple of the identity, when $V_i^H$ and $V_i^{\sigma}$ have the same underlying $U_0 = U_1$-module structure.

  We recall conjugations by the elements $D^H$ on one side and $D^{\sigma}$ on the other
  side both induce the same automorphism $\mathcal{D}$ of $U ⊗ U$.

\begin{prop}\label{cor}
Set for $i = 1,2$ $V_i^H$ a representation of $\UH$ and $V_i^{\sigma}$ a representation of $U_q\gl(1|1)^{\sigma}$ which both restrict to the same irreducible representation of $U=U_0 = U_1$. Then $D^H (D^{\sigma})^{-1} \in \End_{\mathbb{C}}(V_1 ⊗ V_2)$ is a scalar multiple of the identity. 
\end{prop}

\begin{proof}
The density theorem states that if $V$ is a finite dimensional irreducible representation of an algebra $A$ over an algebraically closed field, then $ A \twoheadrightarrow \End(V)$ is surjective. Denote the representations at hand $\rho_{V_i^H}$, $\rho_{V_i^\sigma}$ for $i = 1,2$. We supposed
$$\restriction{\rho_{V_i^H}}{U} = \restriction{\rho_{V_i^\sigma}}{U}.$$
So if $\rho_H =\rho_{V_1^H} ⊗ \rho_{V_2^H}$ and $\rho_{\sigma} = \rho_{V_1^\sigma} ⊗ \rho_{V_2^\sigma}$ we define $ \rho := \restriction{\rho_H}{U⊗ U} = \restriction{\rho_\sigma}{U ⊗ U}$. Using Lemma \ref{L:DH} and Lemma \ref{lem}, for any $x,y \in U$:
$$\rho_H \big(D^H \big) \rho(x ⊗ y) \rho_H \big((D^H)^{-1} \big)=\rho\big( \mathcal{D}(x ⊗ y) \big)= \rho_\sigma \big( D^{\sigma}\big) \rho(x ⊗ y) \rho_\sigma \big((D^{\sigma})^{-1} \big).$$
Which means
$$
\rho_H \big(D^H \big)^{-1} \rho_\sigma \big( D^{\sigma}\big) \rho(x ⊗ y) = \rho(x ⊗ y) \rho_H \big(D^H \big)^{-1}\rho_\sigma \big(D^{\sigma}\big).
$$
Using the density theorem, $\rho_H \big(D^H \big)^{-1}\rho_\sigma
\big(D^{\sigma}\big)$ commutes with any element in\\ {$\End_\C(V_1)
  ⊗ \End_\C(V_2) = \End_\C(V_1 ⊗ V_2)$}. So this linear map is a
scalar multiple of the identity.
\end{proof}

From now on, we consider Hopf algebra $A = \UH \bigotimes_{U} U_q\gl(1|1)^{\sigma}$. 
$A$ contains both algebras $\UH$ and $U_q\gl(1|1)^{\sigma}$.

Formally, setting $q = e^h$, $q^T := e^{hT}$ and $\textbf{i}^{\alpha}
= e^{\textbf{i} \frac{\pi}{2} \alpha}$, we also consider that
$$\textbf{i}^H = k = q^{-{\EE}} \sigma$$
which means that we will only study representations of $A$ that
satisfy this relation. Recall from Equations \eqref{eq:repU} the
representation of $U$ with parameter $a$.
We can look for the representations of $A$ that simultaneously extend
to the representations of $\UH$ and $U_q\gl(1|1)^{\sigma}$ we already
described. If $\varepsilon \in\{0, 1\}$ is the degree of the first
vector $e_0$ of the basis $(e_0,e_1)$ we choose, direct computation of
such a representation $V(\alpha,a,2j,\varepsilon,J)$ shows it is well
defined if and only if:
\begin{equation}
  \left\{
      \begin{aligned}
        (-1)^{\varepsilon} q^{-2j} = a \\
        a = e^{\textbf{i} \frac{\pi}{2} (\alpha+1)} = \textbf{i}^{\alpha+1}\\
      \end{aligned}
    \right.
\end{equation}
Setting $s = q^j \textbf{i}^{\frac{\alpha -3 -2\varepsilon}{2}}=\pm1$, we can compute the coefficient $\RH / R^{\sigma} = D^H / D^\sigma$ given by Proposition \ref{cor} in our case.

\begin{prop}
$\RH / R^{\sigma} = D^H / D^\sigma = s s' (-1)^{\varepsilon \varepsilon '} \emph{\textbf{i}}^{\varepsilon + \varepsilon '} \emph{\textbf{i}}^{\frac{\alpha \alpha ' -1}{2}} q^{j J' + j' J}.$
\end{prop}

\begin{proof}
Using representation $V ⊗ V' = V(\alpha,a,2j,\varepsilon,J)⊗ V(\alpha ',a',2j',\varepsilon ',J')$ in basis\\ $(\e_0 ⊗ e_0, \e_0 ⊗ e_1, \e_1 ⊗ e_0, \e_1 ⊗ e_1)$, we can write:
$$D^H = \textbf{i}^{\alpha \alpha' /2}\begin{pmatrix}
   \textbf{i}^{\frac{\alpha + \alpha ' + 1}{2}} & 0 & 0 & 0 \\
   0 & \textbf{i}^{\frac{-\alpha + \alpha ' - 1}{2}} & 0 & 0 \\
   0 & 0 & \textbf{i}^{\frac{\alpha - \alpha ' - 1}{2}} & 0 \\
   0 & 0 & 0 & \textbf{i}^{\frac{-\alpha - \alpha ' + 1}{2}}
   \end{pmatrix}.$$
Moreover, $D^\sigma = R_1 D'$ and 
$$R_1 = (-1)^{\varepsilon \varepsilon '}\begin{pmatrix}
   1 & 0 & 0 & 0 \\
   0 & (-1)^\varepsilon & 0 & 0 \\
   0 & 0 & (-1)^{\varepsilon '} & 0 \\
   0 & 0 & 0 & (-1)^{\varepsilon + \varepsilon ' +1}
   \end{pmatrix},$$ 
$$D' = q^{-j J' - j' J}\begin{pmatrix}
   q^{-j-j'} & 0 & 0 & 0 \\
   0 & q^{j-j'} & 0 & 0 \\
   0 & 0 & q^{-j+j'} & 0 \\
   0 & 0 & 0 & q^{j+j'}
 \end{pmatrix}.$$ Since $a = \textbf{i}^\alpha = (-1)^{\varepsilon +1}
 \textbf{i} q^{-2j}$, the formulas make appear two square roots of $a$:
$$\sqrt[\alpha]{a} = \textbf{i}^{\alpha/2} \text{ and } \sqrt[j]{a} = \textbf{i}^{\varepsilon + \frac{3}{2}} q^{-j}=s\sqrt[\alpha]{a}.$$
That way, computing any of the diagonal coefficients of $D^H (D^\sigma)^{-1}$ we find the announced element.   
\end{proof}

\section{An integral form of $U_q\gl(n|1)$ and its specialization}
\subsection{Quasitriangular Hopf superalgebra $U_q\gl(n|1)$}
Here we define the $h$-adic quasitriangular Hopf superalgebra
$U_q\gl(n|1)$ that we will use to construct the Links-Gould invariant
$LG^{n,1}$. The conventions we use for generators and relations are
those chosen by Zhang and De Wit in \cite{Zh,DW}. $\mathcal{I} = \{ 1,
2, \ldots , n+1\}$ will be the set of indices. We introduce a grading
$[a]\in \mathbb{Z}/2\mathbb{Z}$ for any $a \in \mathcal{I}$ by setting
$$
[a] = 0 \text{ if } a \leqslant n \text{ and } [a] = 1 \text{ when } a = n+1.
$$
The superalgebra has $(n+1)^2$ generators divided into three
families. There are $n+1$ even Cartan generators $E_a^a$. There are
$\frac{1}{2}n(n+1)$ lowering generators $E_a^b$ parametrized by
$a<b$. Finally there are $\frac{1}{2}n(n+1)$ raising generators
$E_b^a$, with $a<b$. The degree of $E_a^b$ is given by $[a] + [b]$.

For $a \in \mathcal{I}$, $a\neq n+1$, set $K_a=q^{E^a_a}$, and set
$K_{n+1} =q^{-E_{n+1}^{n+1}}$.  In the following $ [X,Y]$ denotes the
super commutator $[X,Y]=XY - (-1)^{[X][Y]} YX$.

Now let us present the relations there are between elements of
$U_q\gl(n|1)$. 

For any $a,b\in\I$ with $|a-b|\ge2$ and for any  $c$ in the interval between $a$ and $b$,
$$
E_b^a = E_c^a E_b^c - q^{\sign(a-b)} E_b^c E_c^a.
$$
For any $a,b\in\I$,
$$
E_a^aE_b^b=E_b^bE_a^a
\text{, }E_a^aE_{b \pm 1}^b = E_{b \pm 1}^b \bp{E_a^a+{\delta_b^a - \delta_{b \pm 1}^a} }
$$
$$
[E_{a+1}^a , E_b^{b+1}] = \delta_b^{a} \frac{K_a K_{a+1}^{-1} - K_a^{-1} K_{a+1}}{q - q^{-1}}$$ 
$$\text{which generalizes for }a<b\text{ to } [E_{b}^a , E_a^{b}] = \frac{K_a K_{b}^{-1} - K_a^{-1} K_{b}}{q- q^{-1}},
$$
$$(E_{n+1}^n)^2 = (E_n^{n+1})^2 = 0 \text{, which implies } (E_{n+1}^i)^2 = (E_i^{n+1})^2 = 0 \text{ for } i< n+1.$$
The Serre relations: for any $a,b\in\I$ with $|a-b|\ge2$,
$$E_a^{a+1} E_b^{b+1} = E_b^{b+1} E_a^{a+1} \text{, } E_{a+1}^a E_{b+1}^b = E_{b+1}^b E_{a+1}^a ,$$
and for $a\le n-1$,
$$E_{a+1}^a E_{a+2}^a =qE_{a+2}^aE_{a+1}^a\ets E_a^{a+1} E_a^{a+2} =qE_a^{a+2}E_a^{a+1}\ets$$
$$E_{a+2}^a E_{a+2}^{a+1} =qE_{a+2}^{a+1}E_{a+2}^a\et 
E_a^{a+2} E_{a+1}^{a+2} =qE_{a+1}^{a+2}E_a^{a+2}.$$ 
These relations can be completed into a set of ``quasi-commutation''
relations indexed by pairs of root vectors (see \cite[Lemma 1]{DW}
where a reordering algorithm gives a constructive proof of the
Poincar\'{e}-Birkhoff-Witt theorem) but these relations are redundant
over the field $\C(q)$.

We consider the Hopf algebra structure given by the
coproduct
$$\Delta(E_{a+1}^a)=E_{a+1}^a⊗K_aK_{a+1}^{-1}+ 1⊗E_{a+1}^a\ets \Delta(E^{a+1}_a)
=K_a^{-1} K_{a+1}⊗E_a^{a+1}+E_a^{a+1}⊗1$$
$$\Delta(K_a)=K_a⊗K_a\et\Delta(E_a^a)=E_a^a⊗1+1⊗E_a^a$$
which admits\footnote{we use here the coproduct and $R$-matrix of
  \cite{KT} conjugated by $D^\gl$.} the universal $R$-matrix $R^\gl
=D^\gl\check R^\gl $ with $D^\gl=q^{\sum_{i\le
    n}E_i^i⊗E_i^i-E_{n+1}^{n+1}⊗E_{n+1}^{n+1}}$ and
$$
\check R^\gl   ={\prod_{i=1}^n\bp{\prod_{j=i+1}^{n}
\e_q((q-q^{-1}) E_j^i ⊗ E_i^j)}\e'_q(E_{n+1}^i⊗E^{n+1}_i )},
$$
where $\e'_q(x)=(1-(q-q^{-1})x)$, $\e_q(x) = \sum_{k=0}^{+\infty}\frac{x^k}{(k)_q!}$, $(k)_q = \frac{1-q^k}{1-q}$ and $(k)_q ! = (1)_q (2)_q \ldots (k)_q$.
Remark that the order of the factors matters in $\check R^\gl$.

\subsection{Integral form and interesting subalgebras}
We now give an integral form of $U_q\gl(n|1)$ which supports evaluation
at $q=-1$.  Let $\A_q$ be the $\Zq$-subalgebra of $U_q\gl(n|1)$ generated
by elements $K_a$, $\mathcal{E}_b^a := \bp{q - q^{-1}} E_b^a$ when $a<b$ and
$\mathcal{E}_b^a := E_b^a$ when $a>b$.  The 
relations of $U_q\gl(n|1)$ 
$$
[E_{b}^a , E_a^{b}] = \frac{K_a K_{b}^{-1} - K_a^{-1} K_{b}}{q - q^{-1}}
$$
for $a<b$, are replaced in algebra $\A_q$ by
$$
[\mathcal{E}_{b}^a , \mathcal{E}_a^{b}] = K_a K_{b}^{-1} - K_a^{-1} K_{b}.
$$
Still, $\A_q$ admits a presentations similar to that of $U_q\gl(n|1)$.  No
additional relations are needed because the analog of the above
commutation relations are enough to express any element in the
Poincar\'{e}-Birkhoff-Witt basis.

In the bosonization $\A^\sigma_q$ of $\A_q$, define for $i = 1, \ldots,n$ the algebra
$$
A_i = \brk{ e_i = -\mathcal{E}_{n+1}^i \sigma ,\, f_i = \mathcal{E}_i^{n+1},\, k_i =  K_iK_{n+1}^{-1} \sigma } \subset \mathcal{A}_q^\sigma.
$$

\begin{prop}
Algebra $A_i$ is isomorphic to $U$. Indeed:
$$
e_i f_i - f_i e_i = k_i - k_i^{-1}, 
$$
$$
k_i e_i + e_i k_i = k_i f_i + f_i k_i = 0.
$$
\end{prop}
\begin{proof}
  Direct computations from the defining relations of $\mathcal{A}_q$
  and Lemma 1 of \cite{DW}.  In particular, $e_if_i - f_i e_i =
  -\mathcal{E}_{n+1}^i \sigma
  \mathcal{E}_i^{n+1}+\mathcal{E}_i^{n+1}\mathcal{E}_{n+1}^i
  \sigma=[\mathcal{E}_{n+1}^i,\mathcal{E}_i^{n+1}]\sigma=k_i - k_i^{-1}.$
\end{proof}

\begin{rem}
  However, $A_i$ is not isomorphic to $U$ as a Hopf algebra (except
  for $A_n$), which can be seen by looking at the coproduct of
  elements of $A_i$ in $\mathcal{A}_q$. This will not be a problem for
  us.
\end{rem}

Set $1 \leqslant i \neq j \leqslant n$. Using \cite{DW} Lemma 1 once
again, we want to see at what conditions any $x \in A_i$ and $y \in
A_j$ commute.
\begin{lemma}
We have the following commutations:
$$
e_i e_j = -q^{-1} e_j e_i \text{, } f_i f_j = -q^{-1} f_j f_i \text{,
} k_i k_j = k_j k_i \text{,}
$$
$$
\text{if }i<j\text{, }e_i f_j - f_j e_i = \sigma K_j K_{n+1}^{-1} \mathcal{E}_j^i \text{, otherwise }e_i f_j - f_j e_i = \sigma \bp{q - q^{-1}} \mathcal{E}_j^i K_{n+1} K_i^{-1},
$$
$$
k_j e_i = -q^{-1} e_i k_j \text{, }k_j f_i = - q f_i k_j.
$$
\end{lemma}
\begin{proof}
  The first two equalities correspond to  \cite[Eq. (38) and (39)]{DW} and the two brackets $[e_i,f_j]$ correspond to  \cite[Eq. (36) (c) and (d)]{DW}.
\end{proof}
\begin{cor}\label{commute}
Setting $q=-1$, in any quotient of $\mathcal{A}_{-1}^{\sigma}$ such that for any $1 \leqslant i<j \leqslant n$, $\mathcal{E}_j^i = 0$, the elements of two distinct $A_i$ commute.
\end{cor}

\subsection{Highest weight representation $V(0^n,\alpha)$}
Let $V(0^n, \alpha)$ be the highest weight irreducible
$2^n$-dimensional representation of $U_q\gl(n|1)$ of weight $(0^n,
\alpha)$, with $\alpha \notin \mathbb{Z}$. So $E_i^i$ is represented
by $0$, except for $E_{n+1}^{n+1}$ that is represented by
$\alpha$. Set $v_0$ a highest weight vector in $V(0^n,\alpha)$ and let
$V_q(0^n,\alpha)=\A_qv_0$. The Poincar\'{e}-Birkhoff-Witt theorem
proves that $$\bigg( \prod_{i=1}^n f_i^{m_i}v_0 \bigg)_{m_i \in
  \{0,1\}}$$ is a basis for vector space $V(0^n,\alpha)$ and for the
free $\Zq$-module $V_q(0^n,\alpha)$.  Set
$\mathcal{A}_{-1}^{\sigma}=\A_q^{\sigma}⊗_{q=-1}\C$ and
$V_{-1}(0^n,\alpha)=V_q(0^n,\alpha)⊗_{q=-1}\C$

\begin{prop}\label{rep}
  In the representation $V_{-1}(0^n,\alpha)$, for any $1
  \leqslant i<j \leqslant n$, $\mathcal{E}_j^i = 0$. 
  So $\mathcal{E}_j^i$ belongs to the kernel $I$ of the representation
  $\mathcal{A}_{-1}^{\sigma}
  \longrightarrow \End(V_{-1}(0^n,\alpha))$. As a consequence, the
  following map is well defined:
  $$
  \begin{array}{rcl}
    \Theta : \bigotimes_{i=1}^n A_i &\longrightarrow& \mathcal{A}_{-1}^{\sigma}/I\\
    ⊗_ix_i&\mapsto&\prod_ix_i
  \end{array}
  .$$
\end{prop}

\begin{proof}
  We want to show that for any basis vector $v\in
  V_{-1}(0^n,\alpha)$ 
  and for $1 \leqslant i<j \leqslant n$, $\mathcal{E}_j^i v = 0$. We
  can write
  $v = f_1^{i_1} \ldots f_n^{i_n} v_0$ where $i_k = 0,1$.\\
  Using \cite{DW} Lemma 1 once more, if\\
  \centerline{\begin{tabular}{rl}
      $c< i$ then& $[E^i_j ,E^{n+1}_c] =0$ by \cite[Eq. (40)]{DW}\\
      $c=i$ then& $[E^i_j ,E^{n+1}_c] =-K_iK_j^{-1}E^{n+1}_j$ by
      \cite[Eq. (36)(a)]{DW}\\
      $i<c<j$ then& $[E^i_j ,E^{n+1}_c]
      =-(q-q^{-1})K_cK_j^{-1}E^i_cE^{n+1}_j$ by \cite[Eq. (43)(a)]{DW}\\
      $j\le c$ then& $[E^i_j ,E^{n+1}_c] =0$ by
      \cite[Eq. (37),(40)]{DW}.
  \end{tabular}}
In all cases, $[\mathcal{E}_j^i, f_c] =[\E^i_j ,\E^{n+1}_c]
=(q-q^{-1})[E^i_j ,E^{n+1}_c] =0$ in $\mathcal{A}_{-1}^{\sigma}$.
  So $\mathcal{E}_j^i v = f_1^{i_1} \ldots f_n^{i_n} (\mathcal{E}_j^i v_0)$. But
  $\mathcal{E}_j^i$ is a raising generator, so $\mathcal{E}_j^i v_0 =
  0$.  Using Corollary \ref{commute}, for $i \neq j$ $A_i$ and $A_j$
  commute in that representation.
\end{proof}

\subsection{$\check R^\gl$ makes sense when $q=-1$}
Here we intend to show that the non diagonal part $\check R^\gl$ of the universal $R$-matrix of $U_q\gl(n|1)$ supports evaluation at $q=-1$, which is not obvious given the formula defining $\check R^\gl$. 
In the bosonization $U_q\gl(n|1)^\sigma$, the universal $R$-matrix is given by 
$$
(R^\gl)^\sigma =D^\gl(\check R^\gl)^\sigma=D^\gl   {\prod_{i=1}^n\bp{\prod_{j=i+1}^{n}
\e_q(\E_j^i ⊗ \E_i^j)}(1+e_i⊗f_i)}.
$$

\begin{prop}\label{P:quasiR} For any $1 \leqslant i < j \leqslant n$, 
  $$\bp{\e_q(\E_j^i ⊗ \E_i^j)-1}V_q(0^n,\alpha)⊗V_q(0^n,\alpha)
  \subset (q+1)\Zq_{loc}V_q(0^n,\alpha)⊗V_q(0^n,\alpha)$$ where
  $\Zq_{loc}$ is the localization of $\Zq$ at $(q+1)$.  Hence
  $(R^\gl)^\sigma$ induces a well defined automorphism of
  $V_{-1}(0^n,\alpha)⊗ V_{-1}(0^n,\alpha)$ where the action of
  $(\check R^\gl)^\sigma$ is given by
  $$ (\check R^\gl)^{\sigma} =
  \prod_{i=1}^n (1 + e_i ⊗ f_i).
  $$
\end{prop}
\begin{proof}
  Define $V=\Zq_{loc}V_q(0^n,\alpha)\subset V(0^n, \alpha)$ so that
  $V_{-1}(0^n,\alpha)\cong V⊗_{q=-1}\C$.  We wish to prove that
  for $1 \leqslant i < j \leqslant n$, in the representation $V⊗V$,
  $\e_q(\E_j^i ⊗ \E_i^j) = 1$ mod $(q+1)$.
  Set $1 \leqslant i < j \leqslant n$. We show by induction on 
  $k \geqslant 1$, that 
  $$\frac{(\E_j^i)^k}{(k)_q !}V \subset (q+1)V.$$ 
  For $k=1$, it follows from $\mathcal{E}_j^i \in I$ (see Proposition
  \ref{rep}). Now we suppose it holds for any $l \in \{1, \ldots ,
  k-1 \}$ and since $\frac{(\E_j^i)^k}{(k)_q
    !}=\frac{(\E_j^i)^{k-1}}{(k-1)_q !}\frac{\E_j^i}{(k)_q}$ it is
  enough to show that $\frac{\E_j^i}{(k)_q}V\subset V$.

  We know that $\E_j^i V\subset (q+1) V$, so $\frac{\E_j^i}{(k)_q}V
  \subset \frac{q+1}{(k)_q} V$.\\
  If $k$ is even, $(k)_q = (q+1) (\frac{k}{2})_{q^2}$ with
  $(\frac{k}{2})_{q^2}=\frac{k}{2}$ mod $(q+1)$ so $\displaystyle{
    \frac{\E_j^i}{(k)_q} V\subset\frac1{(\frac{k}{2})_{q^2}}V=V.}$ If
  $k$ is odd, $(k)_{q} = 1$ mod $(q+1)$ and therefore
  $\displaystyle{\frac{\E_j^i}{(k)_q} V\subset(q+1)V.}$ This concludes
  the proof.
\end{proof}
\section{Links-Gould invariants and the conjecture}

\subsection{Links-Gould invariants $LG^{n,1}$}
The Links-Gould invariants $LG^{n,1}$ are framed link invariants
obtained by applying the modified (one has to use a modified trace, see
\cite{GP}) Reshetikhin-Turaev construction to the ribbon Hopf algebras
$U_q\gl(n|1)^{\sigma}$ we just studied. Like in the Alexander case,
the $R$-matrix can be divided by the value of the twist so that
$LG^{n,1}$ becomes an \emph{unframed} link invariant. Note that this
definition and Viro's work \cite{Vi} show that the first $LG$
invariant $LG^{1,1}$ coincides with the Alexander-Conway polynomial
$\Delta$.

There are several sets of variables used in papers studying $LG$
invariants. Three of them appear regularly: $(t_0,t_1)$, $(\tau, q)$
and $(q^\alpha, q)$. Each set can be expressed in terms of the others
using the following defining relations:
$$
t_0 = q^{-2 \alpha}\text{, }t_1 = q^{2 \alpha +2}\text{,}
$$
$$
\tau = t_0^{1/2} = q^{- \alpha}.
$$

In the case of $LG^{2,1}$, variables $(t_0,t_1)$ nicely lead to a
symmetric Laurent polynomial that has all sorts of Alexander-type
properties \cite{Ish}.

Here we are interested in what happens to $LG^{n,1}$ when you evaluate
$q$ at $-1$, or in other words when you set $t_0 t_1 = 1$.
\subsection{Proof of the conjecture}
Our study of ribbon Hopf algebra $U_q\gl(n|1)^{\sigma}$ allows us to
prove the following, that was conjectured in \cite{DWIL}:

\begin{theo}\label{main}
  For any link $L$ in $S^3$, $LG^{n,1}(L;\tau,-1) =
  \Delta_L(\tau^2)^n$.  This can be translated in variables
  $(t_0,t_1)$:
  $$
  LG^{n,1}(L;t_0,t_0^{-1}) = \Delta_L(t_0)^n.
  $$
\end{theo}
The rest of the section is devoted to proving this identity. First we
identify $V_{-1}(0^n, \alpha)$ as a $⊗_iA_i$-module:
\begin{prop}\label{P:Vfactor}
  Equipped with the action of $⊗_iA_i$ induced by
  $\Theta:⊗_iA_i\to\A_{-1}^{\sigma}/I$, $V_{-1}(0^n, \alpha)$ is isomorphic to
  the irreducible representation $⊗_iV^i$ where each $V^i$ is an
  $A_i$-module isomorphic to the 2-dimensional $U$-module
  $V_{q^{-\alpha}}$.
\end{prop}
\begin{proof}
  By $⊗_iV^i$, we mean the representation 
  $$⊗_i\rho_i:⊗_iA_i\to⊗_i\End_\C(V^i)\cong\End_\C(⊗_iV^i)\text{ where }\rho_i:A_i\to\End_\C(V^i)$$
Set $a=q^{-\alpha}$. For each $i$, $k_i^2$ acts by
  $a^2$ on $V_{-1}(0^n, \alpha)$.  Thus $V_{-1}(0^n, \alpha)$ is a
  representation of the $8^n$-dimensional semi-simple algebra
  $\bigotimes_{i=1}^j\bp{A_i/(k_i^2-a^2)}$.  But for each
  $A_i$, $v_0$ is a highest weight vector of weight $a$. So it belongs to a
  summand of the $\bigotimes_{i=1}^j\bp{A_i/(k_i^2-a^2)}$-module
  $V_{-1}(0^n, \alpha)$ of the form $⊗_iV^i$.  Comparing the
  dimensions which are equal to $2^n$ for both vector spaces, we have
  that $V_{-1}(0^n, \alpha)\simeq ⊗_iV^i$.
\end{proof}

Now we study the action of the pivotal element of $\A_q^{\sigma}$ in
the representation at $q=-1$.
\begin{prop}\label{P:pivot}
  If $K_{2\rho}^\sigma$ is the pivotal element of $\A_q^{\sigma}$, in the
  representation $V_{-1}(0^n, \alpha)$,
  $$
  K_{2\rho}^\sigma  = \Theta(⊗_i \phi_i)
  $$
  where $\phi_i = k_i^{-1} \in A_i$.
\end{prop}
\begin{proof}
  The antipode of $U_q\gl(n|1)$ satisfies
  $S(E^i_{i+1})=-E^i_{i+1}K_{i+1}K_{i}^{-1}$ and
  $S^2(E^i_{i+1})=K_{i}K_{i+1}^{-1}E^i_{i+1}K_{i+1}K_{i}^{-1}=K_{2\rho}E^i_{i+1}K_{2\rho}^{-1}$.
  We can write $K_{2\rho}$ in terms of Cartan generators:
  $$
  K_{2\rho} =K_{n+1}^{n} \prod_{i=1}^n K_i^{n-2i}.
  $$
  Denoting $\langle a|b\rangle := \sum_{i=1}^n a_i b_i - a_{n+1}
  b_{n+1}$, and $\rho$ the graded half sum of all positive roots, we
  find:
  $$
  2 \rho = \sum_{i=1}^n (n-2i)\varepsilon_i + n \varepsilon_{n+1},
  $$
  where $\varepsilon_i$ is the $i^{th}$ basis vector of
  $\mathbb{C}^{n+1}$ and we write any vector $x = \sum_{i=1}^{n+1} x_i
  \varepsilon_i$ in this basis. $K_{2\rho}$ conjugates element $e_i \in A_i$
  as follows:
\begin{alignat*}{2}
   K_{2\rho} e_i K_{2\rho}^{-1} & = & q^{\langle 2 \rho| \varepsilon_i - \varepsilon_{n+1}\rangle} e_i \\
   & = & q^{(n-2i+n)} e_i \\
   & = & q^{2n-2i} e_i.
\end{alignat*}
So if $q=-1$, 
\begin{alignat*}{2}
   \sigma K_{2\rho} e_i K_{2\rho}^{-1}\sigma & = & -e_i\\
   & = & \phi_i e_i \phi_i^{-1} \\
   & = & \Theta(⊗_j \phi_j) e_i \Theta(⊗_j \phi_j^{-1}).
\end{alignat*}
Similarly to Proposition \ref{cor}, we therefore can say that in the
irreducible $⊗_iA_i$-module $V_{-1}(0^n, \alpha)$, $K_{2\rho}^\sigma$
is a scalar multiple of $\Theta(⊗_j \phi_j)$. We call this element $\lambda$. Since the two maps
both act by $q^{n\alpha}$ 
on the highest
weight vector, we conclude that $\lambda = 1$.
\end{proof}

\begin{prop}\label{P:conjDgl} For any $x\in A_i⊗ A_i\subset\A_q⊗\A_q$,
  we have
  $$D^\gl x(D^\gl)^{-1}=\D(x)$$
  where we identified $A_i⊗ A_i\cong U⊗U$.
\end{prop}
\begin{proof}
  By a direct computation,
  $$
  D^\gl E^j_{n+1}⊗1 = E^j_{n+1}⊗1 q^{\sum_{i\le n}
    ((E_i^i+\delta^i_j)⊗E_i^i -(E_{n+1}^{n+1}-1)⊗E_{n+1}^{n+1}} =
  E^j_{n+1}⊗K_jK_{n+1}^{-1}D^\gl
  $$
  Thus $D^\gl e_j⊗1(D^\gl)^{-1}=e_j⊗k_j$.  Similarly $D^\gl
  f_j⊗1(D^\gl)^{-1}=f_j⊗k_j^{-1}$.  Finally $k_i⊗1$ clearly commutes
  with $D^\gl$ and we can conclude using $\tau\circ
  D^\gl=D^\gl\circ\tau$.
\end{proof}

\begin{proof}[Proof of Theorem \ref{main}]
  Let us sum up what we proved up to now to obtain \ref{main}.  Let
  $V_{-1}(0^n, \alpha)\simeq\bigotimes_{i=1}^nV^i$ be the isomorphic
  representations of Proposition \ref{P:Vfactor}. In the following we
  fix such an isomorphism.
  Let $V^i_H$ be a $\UH$-module structure on $V^i$ extending the
  representation of $A_i$.  We therefore obtain $n$ commuting
  R-matrices $R^i=D^i\check R^i$ in
  $\End_\C(V^i⊗V^i)\hookrightarrow\End_\C(V_{-1}(0^n,
  \alpha)⊗V_{-1}(0^n, \alpha))$, where the explicit inclusion maps are
  given by $\iota_i : v\otimes w \mapsto (id^{\otimes i-1} \otimes v
  \otimes id^{\otimes n-i})\otimes (id^{\otimes i-1} \otimes w \otimes
  id^{\otimes n-i})$.  By Proposition \ref{P:quasiR},
  $$\check R^\gl_{|q=-1}=\prod_i \iota_i(\check R^i)\in 
  \End_\C(V_{-1}(0^n, \alpha)⊗V_{-1}(0^n,\alpha)),$$ 
  and by Lemma \ref{L:DH}, Proposition \ref{P:conjDgl}, and the
  density Lemma, the conjugation by $\prod_i \iota_i (D^i)$ is equal to the
  conjugation by $D^\gl$ in $\End_\C(V_{-1}(0^n, \alpha)⊗V_{-1}(0^n,
  \alpha))$.  Hence the braidings on
  $(\bigotimes_{i=1}^nV^i_H)⊗(\bigotimes_{i=1}^nV^i_H)$ and on
  $V_{-1}(0^n, \alpha)⊗V_{-1}(0^n, \alpha)$ are proportional.  Now in
  the process of computing both the Links-Gould invariant and the
  Alexander polynomial, the R-matrices are rescaled by the inverse of
  their twist $\theta^{-1}$ so that the invariants become framing
  independent:
  $$\trace_2(\theta^{-1}(\Id⊗\phi)\tau R)=\Id_{V_{-1}(0^n, \alpha)}$$
  (here $\phi$ denotes any of the pivotal structures which are equal
  by Proposition \ref{P:pivot}).  Hence the rescaled R-matrices
  $R^\gl_{|q=-1} = \prod_i \iota_i (R^H_{V^i⊗V^i})$ and $\bigotimes_iR^H_{V^i⊗V^i}$ are equal up to reordering factors.  Finally, for any braid $\beta\in B_\ell$, the
  associated operators by the Reshetikhin-Turaev construction
  correspond up to reordering as well:
  $$\Psi^\gl_{V_{-1}(0^n, \alpha)^{⊗ \ell}}(\beta)=
  \bp{\Psi^\UH_{V_{-\alpha}^{⊗ \ell}}(\beta)}^{⊗ n}.$$
  At the end, if $\trace_{2,3,...,\ell}\bp{(\Id_{V_{-1}(0^n, \alpha)} ⊗ \phi^{⊗
    \ell-1}) \circ \Psi^\gl_{V_{-1}(0^n, \alpha)^{⊗{\ell}}}(\beta)}=
  d.\Id_{V_{-1}(0^n, \alpha)}$ when\\
  $\trace_{2,3,...,\ell}\bp{(\Id_{V_\alpha} ⊗ \phi_U^{⊗ \ell-1}) \circ
  \Psi^\UH_{V_{\alpha}^{⊗{\ell}}}(\beta)}= c.\Id_{V_\alpha}$, we obtain
  $$
  d = c^n
  $$
  by considering the trace of these two maps. Indeed, the trace is blind to reordering factors.
\renewcommand{\qedsymbol}{\fbox{\ref{main}}}
\end{proof}
\renewcommand{\qedsymbol}{\fbox{\theteo}}

\begin{rem}
  In \cite{GP}, the $LG$ invariant is extended to a multivariable link
  invariant $M(L;q,q_1,\ldots,q_c)$ for links with $c\ge2$ ordered
  components, taking its values in Laurent polynomials
  $\Z[q^{\pm},q_1^\pm,\ldots,q_c^\pm]$.  It is shown in \cite{GP2} that 
  $$LG^{n,1}(\tau,q)=\bp{\prod_{i=0}^{n-1}\frac{q^i}\tau-\frac\tau{q^i}}
  M(L;q,\tau^{-1},\ldots,\tau^{-1}).$$ The proof in this paper should
  adapt to show that
  $$M(L;-1,q_1,\ldots,q_c)=\nabla(q_1,\ldots,q_c)^n$$
  where $\nabla$ is the Conway potential function, a version of the
  multivariable Alexander polynomial.
\end{rem}

\end{document}